\documentclass[a4paper,11pt,fleqn]{article}
\usepackage[T1]{fontenc}
\usepackage{dsfont}
\usepackage{titlesec}
\usepackage[latin5]{inputenc}
\usepackage{amsmath, amsthm, amssymb}
\newtheorem{thm}{Theorem}[section]
\newtheorem{prop}[thm]{Proposition}

\newtheorem{preremark}[thm]{Remark}
  \newenvironment{rmk}%
    {\begin{preremark}\upshape}{\end{preremark}}

\newtheorem{definition}[thm]{Definition}
  \newenvironment{defn}%
    {\begin{definition}\upshape}{\end{definition}}
\numberwithin{equation}{section}

\renewenvironment{proof}{\vspace{1ex}\noindent{\bf Proof. }\hspace{0.5em}}
	{\hfill\qed\vspace{1ex}}

\makeatletter
\renewcommand*\env@matrix[1][*\c@MaxMatrixCols c]{%
  \hskip -\arraycolsep
  \let\@ifnextchar\new@ifnextchar
  \array{#1}}
\makeatother

\titleformat*{\section}{\bfseries}
\titleformat*{\subsection}{\bfseries}
\begin{document}

%\begin{titlepage}

\title{\bf Canonical Forms  for Families of  Anti-commuting Diagonalizable  Linear Operators}

\date{}
\author{
{\bf Yal\c{c}{\i}n Kumbasar\textsuperscript{a}, Ay\c{s}e H\"{u}meyra Bilge\textsuperscript{b,*}}\\
\\
 {\it \footnotesize{\textsuperscript{a} Department of Industrial Engineering, Bogazici University, Bebek, Istanbul, Turkey}}\\
 {\it \footnotesize{\textsuperscript{b} Faculty of Sciences and Letters, Kadir Has University, Cibali, Istanbul, Turkey}}}
  \maketitle

\hrule

\vskip 0.6cm
\noindent{\bf Abstract}\\
\vskip 0.2cm
\baselineskip 12pt
\small{It is well known that a commuting family of diagonalizable
linear operators on a finite dimensional vector space is
simultaneously diagonalizable. In this paper, we consider a family
${\mathcal A}=\{A_a\}$, $A_a:V\rightarrow V$, $a=1,\dots,N$ of  anti-commuting  (complex) linear operators
on a finite dimensional vector space.  We prove that if the family is diagonalizable over the complex numbers, then  $V$ has an
${\mathcal A}$-invariant direct sum decomposition into subspaces
$V_\alpha$
such that the restriction of the family ${\cal A}$ to $V_\alpha$ is a  representation of a Clifford algebra.
Thus unlike the families of commuting diagonalizable operators, diagonalizable anti-commuting families cannot be simultaneously digonalized, but on each subspace, they can be put simultaneously to (non-unique) canonical forms.  The construction of canonical forms
 for complex representations is straightforward, while for the real representations it follows from the results of
[Bilge A.H., \c{S}.
Ko\c{c}ak, S. U\u{g}uz, Canonical Bases for real representations
of Clifford algebras, Linear Algebra and its Applications 419
(2006) 417-439. 3].}

\vskip 0.6cm
\hrule
\vskip 18pt

\let\thefootnote\relax\footnotetext{ * Corresponding author. Tel.: +90212 533 65 32-1349 ; fax: +90212 533 63 30   \\
{\it E-mail addresses:} yalcin.kumbasar@boun.edu.tr, ayse.bilge@khas.edu.tr }
\thispagestyle{empty}
%\pagestyle{empty}
%\end{titlepage}
\nopagebreak[2] %\pagenumbering{roman}

\pagenumbering{arabic}

\normalsize

%\hrule   width16.5cm height0.5pt \vskip 0.7cm
\vskip  1cm
\baselineskip 16pt

\section{Introduction}

\hskip 0.6cm Simultaneous diagonalization of a family of commuting linear operators on a finite
dimensional vector space
is a well known result in linear algebra \cite{Hof}.
This result is applicable to an arbitrary (possibly infinite) family and asserts the existence of a basis
with respect to which all
operators of the family are diagonal.  In this paper,
 we  consider
an anti-commuting  family  ${\cal A}$ of  operators on a finite dimensional vector space $V$
and we show that if the family is diagonalizable over the complex numbers, then the operators in the family can
be put simultaneously into canonical forms
over both  the complex and real numbers.

Real or complex representations of Clifford algebras are typical examples of
anti-commuting  families that are diagonalizable over the complex numbers.  Our  main result is the proof
 that the finite dimensional vector space $V$ has an ${\cal A}$-invariant
direct sum decomposition into subspaces $V_\alpha$, such that, except for the common kernel of the family,
 the restriction of the family to
$V_\alpha$  is either a single nonzero diagonal operator or
a representation of some Clifford algebra of dimension larger than $1$.
This result, presented in Section 3, is derived directly from the fact that if $A_a$ belongs to the family
${\cal A}$, then $A_a$ and $A_a^2$ have the same kernel and the $A_a^2$'s form a commuting diagonalizable family ${\cal B}$, hence they are simultaneously diagonalizable.
One can then diagonalize the family ${\cal B}$ simultaneously, rearrange the basis
and obtain subspaces on which there are finite sub-collections of anti-commuting operators
whose squares are constants, that is, representations of
 Clifford algebras. Note that, since a $1$-dimensional subspace on which there is a single  nonzero operator is a representation of a $1$-dimensional Clifford algebra we could simply state  that $V$ has a decomposition to ${\cal A}$ invariant subspaces on which  the restriction of  the family consist of representations of Clifford algebras.

 In Section 2, we review basic results for commuting families of diagonalizable operators   and we discuss the direct construction of canonical forms, as a motivation for the general  case.  Then, in Section 3, we give the main theorem for the ${\cal A}$ invariant  decomposition of
$V$ for an arbitrary anti-commuting family leading to canonical forms over the complex numbers.  In Section 4, we discuss square diagonalizable operators and  describe the construction of real canonical forms for operators that are  diagonalizable over the complex numbers.

\section{Preliminaries}

\hskip 0.6cm In Section 2.1,   we introduce our notation, present basic properties of commuting families
of diagonalizable operators, and give basic definitions related to Clifford algebras. Then in Section 2.2 we discuss the direct construction of canonical forms.

\subsection{Notation and Basic Definitions}

\hskip 0.6cm In the following $V$ is a finite dimensional real or complex vector space.
Linear operators on $V$ will be denoted by upper case Latin letters $A$, $B$ etc, and the components of their
matrices with respect to some basis
will be denoted by $A_{ij}$, $B_{ij}$. Labels of operators will be denoted by single indices, for example $A_a$, $a=1,\dots,n$ denotes elements of a family of operators.
A family ${\cal A}$ of operators is called an ``anti-commuting family'' if for every distinct pair of operators  $A$ and $B$ in the
family, $AB+BA=0$.
The symbol $\delta_{ij}$ denotes the Kronecker delta, that is $\delta_{ij}=1$, if $i=j$ and zero otherwise.  When we  use
partitioning of matrices, scalars will  denote sub-matrices of appropriate size.

\begin{rmk}
\label{rmk:diag}
If $A$ is diagonalizable operator on a vector space $V$ and $V$ has an $A$-invariant direct
sum decomposition, then the restriction of $A$ to each invariant subspace is also diagonalizable
(Lemma 1.3.10 in \cite{Horn}).  Furthermore if we have a family of commuting (anti-commuting) operators ${\cal A}$ on $V$ and $V$
has an ${\cal A}$-invariant direct sum decomposition, then the restriction of the family to each summand is again a
commuting (anti-commuting) family of diagonalizable operators.
\end{rmk}

Now, we give Theorem~\ref{thm:ComSimDiag} whose proof is adopted from \cite{Horn}.

\begin{thm}
\label{thm:ComSimDiag} Let $\mathcal D$ be a family of
diagonalizable operators on an $n$ dimensional vector space $V$ and $A$, $B$
be in $\mathcal D$.
Then $A$ and $B$ commute if and only if they are simultaneously
diagonalizable.
\end{thm}

\begin{proof}
Assume that $AB=BA$ holds. By a choice of basis we may assume that
$A$ is diagonal, that is
$A_{ij }=\lambda_{i}\delta_{ij}$, $i=1,...,n$.
From the equation $AB=BA$ we have
$$(\lambda_i-\lambda_j)B_{ij}=0,$$  that is $B_{ij}$ is zero unless $\lambda_i=\lambda_j$.
Rearranging the basis, we have a decomposition of $V$ into eigenspaces of $A$. This decomposition is
$B$ invariant, on each subspace $A$ is constant, $B$ is diagonalizable, hence they are simultaneously diagonalizable.

Conversely, assume that $A$ and $B$ are simultaneously
diagonalizable. Then, there is a basis with respect to which their matrices are diagonal.
Since diagonal matrices commute, it follows that the operators $A$ and $B$
commute.
\end{proof}

\begin{rmk}
\label{rmk:finite}
A commuting family of linear operators can be infinite, since we can  always add linear
combinations of the elements of the family.  However, an anti-commuting family is necessarily finite, unless
it contains operators  $A_a$ with $A_a^2=0$. To see this, let
 ${\mathcal A}=\{A_1,\dots ,A_N\}$ be a family
of anti-commuting diagonalizable linear operators on $V$.  That
is
$A_aA_b+A_bA_a=0$,  $a\ne b=1,\dots ,N$.
 The family ${\mathcal A}$  is necessarily linearly independent and it contains $N< n^2$
elements. Because if $B$ is a linear combination of the $A_a$
$a=1,\dots,k$, i.e, $B=\sum_a^k c_a A_a$ and $B$ anti-commutes with
each of the $A_b$'s in this summation it is necessarily zero.
Furthermore, since the anti-commuting family cannot include the
identity matrix, it follows that $N<n^2$.
\end{rmk}

\begin{prop}
\label{prop:ker}
Let $A$ be a linear operator on a finite dimensional vector space $V$. Then $Ker(A)\subseteq Ker(A^2)$.  If $A$ is diagonalizable over $C$, $Ker(A^2)=Ker(A)$. If $A^2$ is diagonalizable over $C$, then  and the restriction of $A$ to the complement of its kernel is diagonalizable over $C$.
\end{prop}

\begin{proof}
The first statement is obvious.  For the second one, we choose a basis with respect to which $A$ is diagonal. Then eigenvalues of $A^2$ are squares
of eigenvalues of $A$. Hence,  $Ker(A^2)=Ker(A)$. To prove the third statement, note that if $A$ is not diagonalizable over $C$,  then in its Jordan form over $C$, there is at least one non-diagonal Jordan block whose square is diagonal. But this is
possible only when the corresponding eigenvalue is zero.
\end{proof}

\begin {rmk}
\label{rmk:sqrcomm}
If $A$ is diagonalizable, then $A^2$ is also diagonalizable. Also if the pair $(A,B)$ anti-commutes then the pairs
$(A,B^2)$ and $(A^2,B^2)$ commute, since
$$AB^2=-B(AB)=B^2A,\quad A^2B^2=A(AB^2)=A(B^2A)=(B^2A)A=B^2A^2.$$
Thus given a family $\{A_1,\dots,A_N\}$
of anti-commuting diagonalizable operators the families $\{A_1,A^2_2\dots,A^2_N\}$ and $\{A^2_1,\dots,A^2_N\}$
are commuting diagonalizable families, hence they are both simultaneously diagonalizable.
\end{rmk}

We give now the definitions related to Clifford algebras and describe briefly the construction of canonical forms for complex representations.

\begin{defn}
Let $V$ be a vector space over the field $k$ and $q$ be a quadratic form on $V$. Then the associative algebra with unit, generated by the vector space $V$ and the identity $1$ subject to the relations
%\begin{equation}
%\label{eq:2.1}
$v\cdot v=-q(v)1$,
%\end{equation}
for any $v\in V$ is called a Clifford algebra and denoted by $Cl(V,q)$.
%If the characteristic of $k$ is not $2$, then the condition  can be replaced by
%\begin{equation}
%\label {eq:2.2}
%v\cdot w+ w\cdot v=-2q(v,w),
%\end{equation}
%for all $v,w\in V$ \cite{Lawson}.
\end{defn}
\begin{defn}
Let $W$ be a real or complex vector space and $C(V,q)$ be  a Clifford algebra.   A representation of $Cl(V,q)$ on $W$ is an algebra homomorphism
$\rho :\quad Cl(V,q)\to End(W).$
\end{defn}

The construction of complex representations of an $N$ dimensional  Clifford algebras on an $n$ dimensional vector space is a  straightforward process. We first construct canonical forms for a pair of anti-commuting operators, then use the requirement that the remaining $N-2$ operators  anti-commute with these two to  show that their  canonical forms are given by the representation of an $N-2$ dimensional Clifford algebra on an $n/2$ dimensional vector space.
If the Clifford algebra contains at least one element with positive square, the construction of real representations is similar to the complex case.  But if all Clifford algebra elements have negative squares,   the construction of simultaneous canonical forms is nontrivial   \cite{Bilge}.

In the next subsection we shall discuss the case of two anti-commuting operators and  outline a proof of Theorem 3.1.

\subsection{A Pair of Anti-commuting Operators: Outline of a Direct Proof}

\hskip 0.6cm The Theorem 3.1. given in Section 3 states that the anti-commuting family is essentially a direct sum of representations of Clifford algebras. The proof of the theorem is almost trivial but it is non-constructive; the difficulty of the construction is in a sense transferred to the construction of the simultaneous canonical forms of Clifford algebras.  In this section we describe in detail the construction of canonical forms for two anti-commuting operators and outline an alternative proof of Theorem 3.1.

Let ${\cal A}=\{A_1,\dots, A_N\}$ be  a finite family of anti-commuting diagonalizable operators.
 One can
always choose a basis with respect to  which any member of the
family is diagonal, hence, without loss of generality we may
assume that $A=A_1$ is diagonal and write
$A_{ij}=\lambda_i\delta_{ij}$.  If $B=B_{ij} $ is any other member of the family,
substituting these in
equation $AB+BA=0$, we obtain
$$
\label{eqn:3.1}
(\lambda_i+\lambda_j)B_{ij}=0
$$
Thus $B_{ij}$ is zero unless  $\lambda_i+\lambda_j=0$.  This happens either when  $\lambda_i$ and $\lambda_j$ are both zero, or when they are a pair of eigenvalues with equal magnitude and opposite sign.
 This suggests that we should group the eigenvalues of $A$
in three sets
$$\{0\},\quad \{\mu_1,\dots,\mu_l\},\quad
\{\lambda_1,-\lambda_1,\dots,\lambda_k,-\lambda_k\},$$ where
$\mu_i+\mu_j\ne 0$ for $i,j=1,\dots,l$.  Let the direct sum of the eigenvectors for each group be $Ker(A)$, $U_A$ and $W_A$. That is
$$V=Ker(A)\oplus U_A\oplus W_A.$$
From the last equation it can easily be seen that these subspaces are ${\cal A}$ invariant and for any other member of the family $B$,
  $B|_{Ker(A)}$ is free and $B|_{U_A}=0$.
On $Ker(A)$ we have a family of  $N-1$ operators and on $U_A$, only $A$ is nonzero. Thus we have nontrivially an $N$ element anti-commuting family only on the subspace $W_A$.

Let $W_{A,i}^\pm$ be the eigenspaces corresponding to the eigenvalues $\pm \lambda_i$ and let
$W_{A,i}=W^+_{A,i}\oplus W^-_{A,i}$.
Since if $AX=\lambda  X$, then $A(BX)=-BAX=-\lambda
(BX)$ and it follows that  and $B$ maps $W^+_{A,i}$ into $W^-_{A,i}$ and vice vera, that is,
$$B(W^+_{A,i})\subset W^-_{A,i},\quad B(W^-_{A,i})\subset
W^+_{A,i},$$ hence $W_{A,i}$'s are $B$ invariant.

If the dimensions of the subspaces $W^\pm_{A,i}$
are not equal, then
the restriction of $B$ to $W_{A,i}=W^+_{A,i}\oplus W^{-}_{A,i}$ is necessarily singular.
Because if $B$ were nonsingular on either $W^\pm_{A,i}$, it would map a linearly independent  set to
a linearly independent set, but this is impossible if the dimensions are different. However,
the restriction of $B$ can be singular even if the dimensions are equal.  On the other hand
if $B$ is nonsingular, then necessarily
$dim(W^+_{A,i})= dim(W^-_{A,i})$
since bases of $W^+_{A,i}$ are mapped to bases of $W^-_{A,i}$ and vice versa.
Thus we can refine the direct sum decomposition of  $W_{A,i}$ and we arrive to the  direct sum decomposition
$$W_{A,i}=(Ker(B)\cap W^+_{A,i})\oplus (Ker(B)\cap W^-_{A,i})
\oplus \tilde{W}^+_{A,i} \oplus \tilde{W}^-_{A,i},$$
where the $\tilde{W}^\pm_{A,i}$ are subspaces of equal dimension on which $B$ is nonsingular.  It follows that the restrictions
of $A$ and $B$ to $W_{A,i}$ have the following block diagonal form
$$A_{|W_{A,i}}=\left(%
\begin{array}{cccc}
  \lambda_i & 0 & 0 & 0 \\
  0 & -\lambda_i & 0 & 0 \\
  0 & 0 & \lambda_i & 0 \\
  0 & 0 & 0 & -\lambda_i \\
\end{array}%
\right),\>
B_{|W_{A,i}}=\left(%
\begin{array}{cccc}
  0 & 0 & 0 & 0 \\
  0 & 0 & 0 & 0 \\
  0 & 0 & 0 & B_1 \\
  0 & 0 & B_2 & 0 \\
\end{array}%
\right),$$
where the first two diagonal blocks may have different dimensions but the last
two diagonal blocks in the restriction of $A$ and the submatrices $B_1$ and $B_2$ in the
restriction of $B$ are square matrices of the same dimension.
Incorporating this decomposition to the previous one,  we have a direct sum decomposition of
$V$ adopted to the pair of anti-commuting diagonalizable operators $A$ and $B$ as
follows
$$V=(Ker(A)\cap Ker(B))\oplus U_A\oplus U_B \oplus W^+_{1}\oplus W^-_{1}\oplus \dots \oplus W^+_{k}\oplus W^-_{k},$$
where $B|_{U_A}=0$, $A|_{U_B}=0$, and both $A$ and $B$ are nonsingular on the $W^\pm_{i}$, and
$dim(W^+_{i})=dim(W^-_{i})$ for $i=1,\dots,k$.
Thus we have now subspaces $W_i=W_i^+\oplus W_i^-$ on which $A^2$ is a nonzero constant and $B$ is nonsingular.

To determine the forms of $B_1$ and $B_2$, we start by the following observation. By Remark \ref{rmk:sqrcomm}
$(A,B^2)$ and $(A^2,B^2)$ are simultaneously diagonalizable and by Proposition \ref{prop:ker} $Ker(B)=Ker(B^2)$.
Recall that  $A$ and $B$ are both nonsingular on the subspace
$W_i$ and  diagonalize $A$ and $B^2$
simultaneously. We choose a basis $\{X_1,\dots,X_m\}$ for $W_i^+$, the $+\lambda$ eigenspace of $A$. Thus
$AX_i=\lambda X_i$, $B^2X_i=\eta_i X_i$, for $i=1\dots,m$
and we define
$Y_i=BX_i$.
Then
$AY_i=A(BX_i)=-B(AX_i)=-\lambda (BX_i)=-\lambda Y_i$,
hence $Y_i$ belongs to the $-\lambda$ eigenspace of $A$.  Furthermore
$BY_i=B^2X_i=\eta_iX_i$.
It follows that with respect to the basis  $\{X_1,\dots,X_m,Y_1,\dots,Y_m\}$,
the matrices of $A$, $B$ and $B^2$ are  as below.

$$
A|_{W_{i}}=\left(%
\begin{array}{cc}
  \lambda	I & 0          \\
  0        & -\lambda I  \\
   \end{array}%
\right),
B|_{W_{i}}=\left(%
\begin{array}{cc}
  0	 & D          \\
  I  & 0          \\
   \end{array}%
\right),
B^2|_{W_{i}}=\left(%
\begin{array}{cc}
  D	 & 0          \\
  0  & D   \\
   \end{array}%
\right),
$$
where all submatrices are square, $I$ is the identity matrix
and $D$ is a diagonal matrix.
If $B^2$ has $q$ distinct eigenvalues $d_1,\dots,d_q$ with eigenspaces of dimensions
$m_i$, we can rearrange the basis so that
$$A|_{W_{i}}=\left(%
\begin{array}{ccccc}
  \lambda	 & 0      & \cdots 	& 0 	     &	0       \\
   0     	 &-\lambda& \cdots  & 0        &  0       \\
   \vdots  & \vdots & \ddots  & \vdots   & \vdots   \\
   0       & 0      & \cdots  & \lambda  &  0       \\
   0       & 0      & \cdots  & 0        & -\lambda \\
    \end{array}%
\right),\>
B|_{W_{i}}=\left(%
\begin{array}{ccccc}
   0	     & d_1      & \cdots 	& 0 	     &	0       \\
   I     	 & 0        & \cdots  & 0        &  0       \\
   \vdots  & \vdots & \ddots  & \vdots     & \vdots   \\
   0       & 0      & \cdots  & 0          &  d_q     \\
   0       & 0      & \cdots  & I          & 0        \\
    \end{array}%
\right),\>
$$
$$
\> B^2|_{W_{i}}=\left(%
\begin{array}{ccccc}
   d_1	   & 0      & \cdots 	& 0 	     &	0       \\
   0     	 & d_1    & \cdots  & 0        &  0       \\
   \vdots  & \vdots & \ddots  & \vdots   & \vdots   \\
   0       & 0      & \cdots  & d_q      & 0        \\
   0       & 0      & \cdots  & 0        & d_q      \\
    \end{array}%
\right),
$$
where $\lambda$ and $d_i$'s denote constant matrices of appropriate size.
 Hence each $W_i$ has a direct sum decomposition
 $$W_i=W_{i,1}\oplus\dots\oplus W_{i,q}$$
 such that both $A^2$ and $B^2$ restricted to each summand are constant matrices.
Thus on the $W_{i,j}$ we have a representation of a two dimensional Clifford algebra.  This is possible in particular
if the dimensions of $W_{i,j}$'s are even.

If $N\ge 2$ and $C_a$ is any  other element of the family then it will be free on the common kernel of $A$ and $B$, and it will be zero on
$U_A$ and $U_B$. From the anti-commutativity of  $C_a$ with $A$ and $B$,
on each $W_{i,j}$, it will be of the form
$$C_a|_{W_{i,j}}=\left(
      \begin{array}{cc}
        0 & -d_jE_a \\
        E_a & 0\\
      \end{array}
    \right)
    $$
where $E_a$ is a square matrix.    By change of basis on $W_{i,j}^+$ we can diagonalize the restriction of  $C_a^2$ to $W_{i,j}^+$, rearrange the basis to obtain subspaces on which $A$, $B$ and $C^2_a$'s are constants.  Here, by construction,  the nonzero
the eigenvalues of $C_a$ occur in pairs with equal absolute value and opposite sign and the dimensions of the zero eigenspaces in $W_{i,j}^+$ and $W_{i,j}^-$ are the same.

It is possible to continue with this construction  by adding operators one by one, but this procedure gets more and more complicated. Instead, we note that the $C_a^2$'s restricted to $W_{i,j}^+$ form a commuting diagonalizable family, and instead of adding new operators one by one, we could diagonalize them simultaneously, and find subspaces on which $C_a^2$'s are constants. This remark suggests that one could do the same trick at the beginning and diagonalize simultaneously the squared family. This leads to the proof given below.

\section{Anti-commuting Families of Diagonalizable Linear Operators: Proof}

\hskip 0.6cm The construction of simultaneous canonical forms for a family   ${\cal A}=\{A_1,\dots,A_N\}$,
  of anti-commuting of diagonalizable operators on a finite dimensional (real or complex)   vector space $V$,
   is
based on the fact that the family of squared operators ${\cal B}=\{A_1^2,\dots,A_N^2\}$  is a commuting diagonalizable family, hence it is simultaneously diagonalizable.
One can then find a
basis with respect to which the  family  of squared operators  is diagonal and rearrange
 this basis in such a way that
the vector space $V$  is a direct sum of subspaces on which the operators of the family  ${\cal B}$ are either zero or constant.  This is nothing but a representation of a Clifford algebra which exists in  specific dimensions.
We state this result below.

\begin{thm}
Let ${\cal A}$ be a family of $N$ diagonalizable anti-commuting operators on an $n$ dimensional vector space $V$. Then $V$ has a direct sum decomposition
$$V=U_0\oplus U_1\oplus \dots U_j\oplus W_2^1 \oplus\dots \oplus W_2^{k_2}\oplus\dots \oplus W_N^1\oplus \dots \oplus W_N^{k_N},$$
where $U_0$ is the common kernel of the family,  $U_i$'s are subspaces of arbitrary dimensions on which only $A_i$ is nonzero and it is nonsingular with $j\le N$,
and $W_i^j$'s are subspaces on which the restriction of the family ${\cal A}$ is a representation of an $i$ dimensional Clifford algebra with $k_i\le {N\choose i}$
\end{thm}

\begin{proof}
Let ${\cal B}=\{A_1^2,\dots,A_N^2\}$ be the family of squared operators.  Since ${\cal A}$ is an anti-commuting family, ${\cal B}$ is a commuting diagonalizable family hence it is simultaneously diagonalizable. Let $\{X_1,\dots ,X_n\}$
be a basis with respect to which ${\cal B}$ is diagonal.  We can group the eigenvectors in such a way that on the subspace spanned by each group the $A_i^2$'s are constant.  It follows that on each of these subspaces the $A_i$'s belong to a representation of a Clifford algebra.
\end{proof}

As an example consider a family of $N=5$ anticommuting diagonalizable operators $A_i$ on a $n=20$ dimensional vector space and let
$\{X_j\}_{j=1}^20$ be a basis with respect to which the squared family is diagonal. The ranges of the operators $A_i$ are given below.
\begin{eqnarray}
&&{\rm Range}(A_1)={\rm Span}\{X_5,X_6,X_8,X_{12},X_{16},X_{19}\},\cr
&&{\rm Range}(A_2)={\rm Span}\{X_3,X_7,X_9,X_{11},X_{13},X_{15},X_{17}\},\cr
&&{\rm Range}(A_3)={\rm Span}\{X_2,X_4,X_5,X_6,X_8,X_{10},X_{12},X_{14},X_{16},X_{19}\},\cr
&&{\rm Range}(A_4)={\rm Span}\{X_1,X_3,X_5,X_7,X_{10},X_{11},X_{12},X_{13},X_{14},X_{15},X_{16},X_{17},X_{19},X_{20}\},\cr
&&{\rm Range}(A_5)={\rm Span}\{X_2,X_4,X_5,X_6,X_{8},X_{10},X_{12},X_{14},X_{16},X_{19}\},
\nonumber
\end{eqnarray}
 We can see that the common kernel of the family is
 $U_0={\rm Span}\{X_{18}\}$.
 There are two  subspaces on which there is a single nonzero operator. These are
 $U_1={\rm Span}\{X_1,X_{20}\}$
 where $A_4$ is nonzero and
 $U_2={\rm Span}\{X_9\}$
 where $A_2$ is nonzero. Since all other operators are zero,
 $A_4(X_1)=\mu_1,\quad A_4(X_{20})=\mu_2$,
 where $\mu_1$ and $\mu_2$ are arbitrary.
 There are $2$  subspaces on which there are two nonzero operators. The first one is
 $W_2^1={\rm Span}\{X_2,X_{4}\}$
 where $A_3$ and $A_5$ are nonzero. On this subspace we have a representation of a two dimensional Clifford algebra.
 Thus
 $A_3(X_2)=\lambda_1X_2,\quad A_3(X_4)=-\lambda_1X_4$, and
 $A_5(X_2)=X_4,\quad A_5(X_4)=dX_2$.
 The other subspace on which there is a representation of a $2$ dimensional Clifford algebra is  $W_2^2={\rm Span}\{X_3,X_{7},X_{11},X_{13},X_{15},X_{17}\}$
 where $A_2$ and $A_4$ are nonzero. This is a six dimensional subspace, the eigenvalues of  $A_2$ and $A_4$ can be the same or different. Hence we may have a combination of reducible and irreducible representations.
There are two subspaces with three nonzero operators. These are
$W_3^1={\rm Span}\{X_6,X_{8}\}$
 on which $A_1$, $A_3$ and $A_5$ are nonzero and
 $W_3^1={\rm Span}\{X_{10},X_{14}\}$
 on which $A_1$, $A_3$ and $A_5$ are nonzero.
 On these subspaces we have representations of three dimensional Clifford algebras.
 Finally on
 $W_4^1={\rm Span}\{X_{5},X_{12},X_{16},X_{19}\}$
 the operators $A_1$, $A_3$, $A_4$ and $A_5$ are nonzero.  There is a representation of  a $4$  dimensional Clifford algebra which can exist only on a four dimensional subspace.

\section {Square Diagonalizable Anti-commuting Families of Linear Operators}

\hskip 0.6cm The construction in the previous section suggests that we may only require the diagonalizability of the squared family in order to obtain canonical forms for an anti-commuting family.  This will not be quite true, because there will be difficulties when the kernels of the operators in the original and the squared family are different. In this section we consider now an anti-commuting family ${\cal A}$ of operators whose squares are diagonalizable.

In the easiest case, the family ${\cal A}$ is a family of real operators that are diagonalizable over $C$ but not diagonalizable over $R$.
In this case $A^2$ is diagonalizable over $R$ and
 $Ker(A)=Ker(A^2)$. The  construction above works with the exception that the dimensions of the representation spaces are determined by the real representations of Clifford algebras  \cite{Bilge}. We have thus the analogue of Theorem 3.1, where the only difference is that  the representations of the Clifford algebras are real.

\begin{thm}
Let ${\cal A}$ be a family of $N$  real, anti-commuting  operators
on an $n$ dimensional vector space $V$.  If the operators $A_a\in {\cal A}$ are diagonaliazable over $C$, then $V$ has a direct sum decomposition
$$V=U_0\oplus U_1\oplus \dots U_j\oplus W_2^1 \oplus\dots \oplus W_2^{k_2}\oplus\dots \oplus W_N^1\oplus \dots \oplus W_N^{k_N},$$
where $U_0$ is the common kernel of the family,  $U_i$'s are subspaces of arbitrary dimensions on which only $A_i$ is nonzero and it is nonsingular with $j\le N$,
and $W_i^j$'s are subspaces on which the restriction of the family ${\cal A}$ is a real representation of an $i$ dimensional Clifford algebra with $k_i\le {N\choose i}$.
\end{thm}

If $Ker(A)\ne Ker(A^2)$, then in the  minimal polynomial $m_A(t)$ of $A$, the only nonlinear factor is $t^2$.
As in the previous section, we can diagonalize the squared family, arrange the eigenspaces so that $V$ is a direct sum of subspaces on which $A_a^2$'s are constant.  But as opposed to the previous case,  if $A_a^2$ is zero on some subspace, $A_a$ need not be zero, hence we may have a family of anti-commuting matrices whose squares are positive, negative or zero.  This is just the representation of some degenerate Clifford algebra \cite{DKL} for which the construction of canonical forms is not known.

\section*{Acknowledgements}
\hskip 0.6cm This work is based on the M.Sc. Thesis of the first author, presented at the Institute of Science and Technology at Istanbul Technical University, May 2010.

\vskip 18pt
\hskip -6mm 
\end{document}